\documentclass[11pt]{amsart}

\usepackage[latin1]{inputenc}
\usepackage{amssymb}
\usepackage{latexsym}
\usepackage{longtable}

\newcommand{\Q}{\mathbb{Q}}
\newcommand{\Z}{\mathbb{Z}}
\newcommand{\F}{\mathbb{F}}

\newcommand{\C}{\mathbb{C}}

\newcommand{\Jac}{{\rm Jac\,}}

\newfont{\gotip}{eufb10 at 12pt}

\newcommand{\gp}{\mbox{\gotip p}}

\newtheorem{proposition}{Proposition}
\newtheorem{theorem}{Theorem}

\begin{document}

\title{Computations on Modular Jacobian Surfaces}

\author{Enrique Gonz\'alez--Jim\'enez}
\address{Department de Matem\`atiques, Universitat Aut\`onoma de Barcelona,\\
E-08193 Bellaterra, Barcelona, Spain}

\author{Josep Gonz\'{a}lez}
\address{Escola Universit\`aria Polit\`ecnica de Vilanova i la Geltr\'{u},\\
Av. V\'{\i}ctor Balaguer s/n, E-08800 Vilanova i la Geltr\'{u}, Spain}

\author{Jordi Gu\`{a}rdia}
\address{Escola Universit\`aria Polit\`ecnica de Vilanova i la Geltr\'{u},\\
Av. V\'{\i}ctor Balaguer s/n, E-08800 Vilanova i la Geltr\'{u}, Spain}

\thanks{The first author was supported in part by DGI Grant BHA2000-0180. The second author was supported in part by DGI Grant BFM2000-0794-C02-02. The third author was supported in part by DGICYT Grant BFM2000-0627.}

\maketitle

\begin{abstract}  
We give a method for finding rational equations of genus 2 curves whose jacobians are  abelian varieties $A_f$ attached by Shimura to  normalized newforms  $f \in S_2( \Gamma_0(N))$. We present all the curves corresponding to  principally polarized surfaces $A_f$ for $N\leq500$.
\end{abstract}

\section{Introduction}
Given   a normalized newform  $f=\sum_{n>0}a_n q^n \in S_2( \Gamma_0(N))$, Shimura \cite{shimura71}-\cite{shimura73} attaches to it an  abelian variety $A_f$ defined over $\Q$ of dimension  equal to the degree of the number field $E_f=\Q(\{ a_n\})$. The  Eichler-Shimura congruence makes it possible to compute at every  prime $p\nmid N$  the characteristic polynomial of the Frobenius endomorphism    acting on the Tate module  of $A_f/\F_p$ from the coefficient $a_p$ and its Galois conjugates.
In consequence, when $A_f$ is $\Q$-isogenous to the jacobian  of a curve $C$ defined over $\Q$,   the
number of points of the reduction of this curve  mod a prime  $p$ of good reduction can be obtained from the characteristic polynomial of the Hecke operator $T_p$ acting on $H^0(A_f,\Omega^1)$. Among these {\it jacobian-modular curves}, those  which are  hyperelliptic of low genus  are especially interesting for public key cryptography.

As an optimal quotient of the jacobian of $X_0(N)$,   $J_0(N)$, the abelian variety $A_f$   has    a natural  polarization  induced from $J_0(N)$. 
We will focus our attention on polarized surfaces $A_f$ which are $\Q$-isomorphic to jacobians of genus $2$ curves. Wang \cite{wang95} gave a first step in the determinations of such curves. More precisely, using modular symbols he computed the periods of  $f$ and its Galois conjugate    and    presented  $A_f$ as a complex torus with an explicit polarization. For those  principally polarized $A_f$, Wang   computed numerically Igusa invariants by means of even  Thetanullwerte and built an hyperelliptic curve $C/\Q$ such that $\Jac C\simeq A_f$ over $\overline{\Q}$. The curves $C$ obtained with this procedure have two drawbacks: they have huge coefficients, and, moreover, we only know that their jacobians are $\overline{\Q}$-isomorphic to the corresponding abelian varieties  $A_f$, but we don't know whether they are $\Q$-isomorphic, or even $\Q$-isogenous. Frey and Muller  \cite{frey:muller97}  looked for  a curve $C'/\Q$ among the twisted curves of $C$ such that 
the local factors of  the $L$-series of $\Jac C'$ and $A_f$ 
 agree for all primes less than a large enough bound.

In this paper  we want to go one step further in the
determination of these  jacobian modular surfaces. We describe  a more  arithmetical and 
efficient method, based on odd Thetanullwerte, which  solves the problem up to numerical approximations. Our method provides equations  $C_F:y^2=F(x)$ with $F(x)\in \Q[x]$
such that $\Jac C_F$ or $\Jac C_{-F}$ is  $A_f$. The  sign is chosen using  the Eichler-Shimura congruence.

We  have  implemented  a 
program in {\sc Magma}  to determine modular jacobian surfaces  and  equations for the corresponding curves. We have found all the modular jacobian surfaces of level $N\leq 500$. The equations obtained for the corresponding curves  are presented at the end of the paper. It is remarkable   that almost all of them  are minimal equations over $\Z [1/2]$.

\section{Theoretical foundations}

A polarized abelian variety $(A,\Theta)$  of dimension $g$ defined over  $\C$ can be realized as a complex torus $T_A=\C^g/\Lambda$, where $\Lambda$ is the  period lattice of $A$ with respect to a basis of $H^0(A,\Omega^1)$, with a nondegenerate Riemann form defined on $\Lambda$. We choose a symplectic basis for $\Lambda$, and write it as a $2g\times g$ matrix $\Omega=(\Omega_1|\Omega_2)$. The normalized period matrix
$Z=\Omega_1^{-1}\Omega_2$ satisfies the Riemann conditions $Z=\,{}^tZ$, $Y=\mbox{Im}Z$ is positive definite and  the Riemann theta function:
$$
\theta (z):=\theta (z;Z):=\sum_{n\in {\Z}^{g}}\exp (\pi
i^{t}n.Z.n+2\pi i ^{t}n.z)
$$
is holomorphic in $\C^g$. The values of the Riemann theta function at
2-torsion points are called Thetanullwerte. Historically, only
the even Thetanullwerte, i.e., the values of the theta function
at even 2-torsion points have been studied, since the values at
odd 2-torsion points are always zero. Anyway, the values of the
derivatives of the theta function at the odd 2-torsion points have
nice properties, and also do provide useful geometrical
information (\cite{guardia00}).

We now give the theoretical results which allow  one to recognize when a principally polarized abelian surface is the jacobian of a genus 2 curve.

\begin{proposition}
Let $(A,\Theta)$ be an irreducible principally polarized abelian
surface defined over  a number field $K$. There exists a
hyperelliptic curve $C$ of genus 2 defined over $K$ such that
$A=\Jac C$.
\end{proposition}
\noindent
{\bf Proof:} 
It is well known that the irreducibility of $A$ implies that
$A=\Jac C$ for a certain hyperelliptic curve $C$ defined over
$\mathbb{C}$. But for genus 2 curves, the Abel-Jacobi map in
degree 1 is an isomorphism between the curve $C$ and the $\Theta$
divisor in $\Jac C=A$. Hence, we can assume that $C=\Theta$, which
is defined over $K$.\qed

\begin{proposition}
\label{irred-g=2} A principally polarized abelian surface
$(A,\Theta)$ is not
 irreducible if and only if there is an even
2-torsion point $P$ such that the corresponding even
Thetanullwerte vanishes.
\end{proposition}
\noindent {\bf Proof:}
If $(A,\Theta)$ is irreducible  principally
polarized, then it is isomorphic to the jacobian of a
hyperelliptic genus 2 curve, and hence every even Thetanullwerte
is non-zero.

 Conversely, assume that $(A,\Theta)$ is
the product of two elliptic curves $E_1,
E_2$. This means that
the theta function $\theta_A$ associated to the pair $(A,\Theta)$
is equal to $\theta_1\theta_2$, where we denote by $\theta_i$ the
theta function associated to the elliptic curve $E_i$. Let $O_i$
be the
 zero point in $E_i$, which is  the unique odd 2-torsion
point in $E_i$. The pair $O=(O_1,O_2)\in E_1\times E_2$ gives an
even two torsion point in $A$, which satisfies $\theta_A(O)=0$.\qed

Once we know that a principally polarized abelian surface $A$ is a
jacobian, we want a method to find a curve $C$ such that $A
\simeq\Jac C$. We would like to be careful enough to assure that,
when $A$ is defined over a number field $K$, the curve $C$ and the
isomorphism $A \simeq\Jac C$ are also defined over $K$. 
 The following result, which can be found in
 \cite{guardia00}, will be basic  for our purpose.

\begin{theorem}
\label{solucio-sextica}Let $
F(x)=a_{6}X^{6}+a_{5}X^{5}+\dots+a_{1}X+a_{0}\in \C[X] $ be a
separable polynomial of degree 5 or 6. Let $\Omega =(\Omega
_{1}|\Omega _{2})$ be the period matrix of  the hyperelliptic
curve $ C_{F}: Y^2=F(X)$ with respect to the basis
$\displaystyle\omega _{1}=\frac{dx}{y}$,   $\displaystyle\omega _{2}=%
\frac{xdx}{y}$ of $H^{0}(C_{F,}\Omega ^{1})$ and  any
symplectic
basis of $%
H_{1}(C_{F},\mathbb{Z})$, and take $Z_{F}=\Omega _{1}^{-1}\Omega
_{2}$.
\begin{itemize}
\item[a)] The   roots $\alpha_k$ of the polynomial $F$
are the ratios $\displaystyle \frac{x_{k,2}}{x_{k,1}}$, given by
the solutions $(x_{k,1},x_{k,2})$ of the six homogeneous linear
equations
$$
\left(
\begin{array}{ll}
\displaystyle\frac{\partial \theta }{\partial z_{1}}(w_{k}) &
\displaystyle\frac{\partial \theta }{%
\partial z_{2}}(w_{k})
\end{array}
\right) \Omega _{1}^{-1}\left(
\begin{array}{c}
X_{1} \\
X_{2}
\end{array}
\right) =0,
$$
where $w_1,\dots,w_6$ are the six odd 2-torsion points of
$J(C_{F})$, given by
$$
\begin{array}{ll}
w_{1}=\frac{1}{2}Z_{F}{\left(
\begin{array}{l}
0 \\
1
\end{array}
\right)} +\frac{1}{2}{\left(
\begin{array}{l}
0 \\
1
\end{array}
\right)} , & w_{2}=\frac{1}{2}Z_{F}{\left(
\begin{array}{l}
0 \\
1
\end{array}
\right)} +\frac{1}{2}{\left(
\begin{array}{l}
1 \\
1
\end{array}
\right)} , \\
w_{3}=\frac{1}{2}Z_{F}\left(
\begin{array}{l}
1 \\
0
\end{array}
\right) +\frac{1}{2}\left(
\begin{array}{l}
1 \\
0
\end{array}
\right) , & w_{4}=\frac{1}{2}Z_{F}\left(
\begin{array}{l}
1 \\
0
\end{array}
\right) +\frac{1}{2}\left(
\begin{array}{l}
1 \\
1
\end{array}
\right) , \\
w_{5}=\frac{1}{2}Z_{F}\left(
\begin{array}{l}
1 \\
1
\end{array}
\right) +\frac{1}{2}\left(
\begin{array}{l}
0 \\
1
\end{array}
\right) , & w_{6}=\frac{1}{2}Z_{F}\left(
\begin{array}{l}
1 \\
1
\end{array}
\right) +\frac{1}{2}\left(
\begin{array}{l}
1 \\
0
\end{array}
\right).
\end{array}
$$
When $\deg F=5$, one of these ratios is infinity and we discard it.
\item[b)] Let $W_j=(\alpha_j,0)$ be the 
Weierstrass point corresponding to $w_j$. Denote by
$H[W_{j}]$ the hyperplane of $\mathbb{P}^{1}$ given by the
equation
$$H[W_{j}](X_{1},X_{2}):=\left(
\begin{array}{ll}
\displaystyle\frac{\partial \theta }{\partial z_{1}}(w_j) &
\displaystyle\frac{\partial \theta }{\partial z_{2}}(w_j)
\end{array}
\right) \Omega _{1}^{-1}\left(
\begin{array}{c}
X_{1} \\
X_{2}
\end{array}
\right).
$$
The discriminant $\Delta _{alg}(C_{F})$ of the polynomial $F$ satisfies the relation
$$
\begin{array}{ll}
\Delta _{alg}(C_{F})^{7}=2^{120}a_{6}^{10}\pi ^{60}\det \Omega
_{1}^{-30}\prod_{j<k}H[W_{j}](1,\alpha _{k})^{2} & \mbox{if
$\deg(F)=6;$}\\
\\
 \Delta _{alg}(C_{F})^{5}=2^{80}a_{5}^{10}\pi ^{80}\det \Omega
_{1}^{-20}\prod_{j<k}H[W_{j}](1,\alpha _{k})^{2} & \mbox{if $\deg(F)=5.$}\\
\end{array}
$$
\end{itemize}
\end{theorem}

\section{Determination of hyperelliptic equations}

We explain here how one can, 
given  an irreducible 
abelian surface $(A, \Theta)$ defined over $K$, look  for a hyperelliptic curve $C_F:Y^2=F(X)$   such that 
$A$ is $K$-isomorphic to $\Jac C_F$. We have divided our method into four steps.

\vskip 0.2cm
\noindent
{\bf Step 1: Period matrix.}
The first step consists in choosing a suitable period matrix $\Omega$ for $A$. We have to fix  a symplectic basis of $H_1(A,\Z)$,   a convenient basis of $H^0(A,\Omega^1_{A/K})$ and  compute the corresponding period matrix.
 The following result assures us that the basis of regular differentials  can be chosen  arbitrarily.
\begin{proposition}(\cite{gogo00}).
Let $C/K$ be a genus $2$ curve. For  every linearly independent
pair of regular differentials  $\omega_1,\omega _2\in
H^0(C,\Omega^1_{C/K})$, there exists a polynomial $F(X)\in K[X]$ of degree $5$
or $6$ without double roots such that the functions on $C$ given
by
$$
x=\frac{\omega_1}{\omega_2}\,,\quad  y=\frac{ dx}{\omega_2}
$$
satisfy the equation $y^2= F(x)$.
\end{proposition}

\noindent
{\bf Step 2: Weierstrass points.}  In this step, we compute  the roots
$\alpha_k$ of the polynomial $F$ given by the first part of the theorem
\ref{solucio-sextica}, and we take  the   monic polynomial
$F_0(X)=\prod_k(X-\alpha_k)\in K[X]$.

\vskip 0.2cm
\noindent
 {\bf Step 3: Leading
coefficient.}
 With the formulas given for the discriminant in part  {\it b)} of theorem \ref{solucio-sextica}, we
obtain  $a_6^{10}\in K$ (or $a_5^{10}\in K$ if $\deg F_0=5$). We choose one of the tenth roots  $a_6'\in K$ of this value 
and take  the polynomial $F_1(X)=a_6' F_0(X)\in K[X]$. 

\vskip 0.2cm
\noindent
{\bf Step 4: Hyperelliptic equation.}
At this point,   it only remains to find the tenth root of unity $\zeta$ such that  $F=\zeta F_1$. 
Since  the curves  $C_F$ and $C_{\lambda^2 F}$ with $\lambda\in K^{*}$ are $K$-isomorphic,  it suffices to consider only the cases $\zeta =1$ and $\zeta=-1$, when $-1\notin K^2$.  First we  check  whether  $C_{F}$ and $C_{-F}$ are $K$-isomorphic. If they are not,  then we look if $\Jac C_F$ and $\Jac C_{-F}$ are not 
$K$-isogenous.  In this case,  by Faltings Theorem, only one of their $L$-series will agree with the $L$-series of $A$ and this will give the right sign for $F=\pm F_1$. 
 In fact, it will  suffice  to find   a prime $\gp$ in $K$ 
 of good reduction for the  curves $C_F$ and $C_{-F}$
such that their reductions mod $\gp$ have a different number of points.

In the case that $C_F$ and $C_F$ are not $K$-isomorphic and $\Jac C_F$ and $\Jac C_{-F}$ are $K$-isogenous, we cannot determine the right sign. Anyway, we know that both jacobians $\Jac C_F$ and $\Jac C_{-F}$  are $K(\sqrt{-1})$-isomorphic to $A_f$, and one of them is
$K$-isomorphic.

\section{Modular computations}

We  apply  the method described in the previous section to 
 present the irreducible principally polarized 
two-dimensional factors of $J_0(N)^{\mbox{new}}$ as jacobians of curves, for $N\le 500$.

In order to do this, we begin
looking for the normalized newforms $f=\sum a_nq^n\in S_2(\Gamma_0(N))$ such
that the number field $E_f=\Q(\{a_n\})$ is quadratic. For each of these
newforms, we 
 take an integral basis of the $\C$-vector space generated by
$f$ and its Galois conjugate ${}^\sigma f$.  We also determine a symplectic basis of
$H_1(A_f,\Z)$. If $A_f$ is principally polarized, we compute the period matrix
with respect to these bases, using  the
package on modular symbols written by W. Stein in {\tt Magma}. 

Next, we check the irreducibility of $A_f$ by means of proposition \ref{irred-g=2}. We remark that all the  $A_f$ studied are irreducible. 

We now 
apply the method of section 3. 
We  follow the steps  described there, to find the corresponding
curves $C_F: Y^2=F(X)$.  Since we are working over $\Q$, we can change the polynomial $F(X)$
in order to obtain an integral equation. We multiply  $F(X)$
 by  $d=t/b$, where $t\in\Z$ is the square of the l.c.m.
of the denominators of the coefficients of $F$, and $b\in\Z$ is the g.c.d. of
their numerators divided by its maximum square-free factor. 
It is worth remarking that the equations obtained have very small coefficients, even before finding  the integral model.

The only case in which we have found a curve $C_F$ such that $\Jac C_F$ and 
$\Jac C_{-F}$ are $\Q$-isogenous occurs for $N=256$, but in fact both curves are already $\Q$-isomorphic, because  the corresponding polynomial $F(X)$  is odd.

We have used three tests to check the correctness of our equations. First, we have computed the absolute Igusa invariants of the curves $C_F$ in two different ways: algebraically from the coefficients of our equations, and numerically from the even Thetanullwerte of the period matrix. They have agreed to high accuracy in all cases.  Second, we have compared the local factors of the $L$-series of $\Jac C_F$ and $A_f$ for all primes $p<100$ not dividing $\Delta_{alg}(C_F)$.
Finally,  we have computed the odd part of the conductor of $C_F$ using  the program {\tt genus2reduction} by Q. Liu.  In all cases, this
odd part agreed with the odd part of the square of the level of the newform $f$, as it should by \cite{carayol86}.  It is worth noting that in almost all cases
our equations are minimal over $\Z[1/2]$.


We  illustrate our computations with an example. The first level for which
$J_0(N)^{\mbox{new}}$ has a proper two-dimensional factor is $N=63$. Using
{\tt Magma} we identify the corresponding normalized newform $f$:
$$
f=
q+\sqrt{3}q^2+q^4-2\sqrt{3}q^5+q^7-\sqrt{3}q^8-6q^{10}+
2\sqrt{3}q^{11}+2q^{13}+\dots
$$
An integral basis of the space $\langle f, {}^\sigma f\rangle$
is 
$$
f_1=q + q^4 + q^7 - 6q^{10} + 2q^{13} +\dots, \qquad \qquad
f_2=q^2 - 2q^5 - q^8 + 2q^{11}+\dots
$$ 
A basis for $H_1(A_f,\Z)$ in terms of modular symbols is given by
$$
\begin{array}{lll}
\gamma_1 &=&  \{-\frac1{24}, 0\} -\{-\frac1{28}, 0\} + \{-\frac1{30}, 0\} -\{-\frac1{51}, 0\} -\{-\frac13,
-\frac27\},
\\ \\
\gamma_2&=&    \{-\frac1{24}, 0\}  -\{-\frac1{28}, 0\} + \{-\frac1{39}, 0\}  -\{-\frac1{57}, 0\} -\{-\frac16,
-\frac17\},
\\ \\
\gamma_3&=&    \{-\frac1{24}, 0\} + \{-\frac1{39}, 0\}  -\{-\frac1{45}, 0\}  -\{-\frac1{60}, 0\}  -\{-\frac13,
-\frac27\} -\{\frac37, \frac49\},     
\\ \\
\gamma_4&=&\{-\frac{1}{36}, 0\} -\{-\frac1{49}, 0\} + \{-\frac1{51}, 0\} -\{-\frac1{54}, 0\} + \{-\frac1{57}, 0\}\\
\\
\multicolumn{3}{r}{-\{-\frac1{60}, 0\}  - \{-\frac1{3}, -\frac27 \}}.    
\end{array}
$$  

Computing the intersection matrix of these paths we see that $A_f$ is principally polarized. 
We find a symplectic basis for $H_1(A_f,\Z)$, and compute the periods of $f_1,
f_2$ with respect to these bases. We obtain as period matrix
$\Omega=(\Omega_1 \mid \Omega_2)$ for  $A_f$: 
$$
\begin{array}{c}
\Omega_1=\left(
\begin{array}{cc}
0.3590439\dots +i* 0.6218823 \dots & -2.2150442\dots + i* 1.2788564\dots \\
-2.2150442\dots + i* 3.8365691\dots & 1.0771318\dots+i*0.6218823\dots
\end{array}\right),
\\
\\
\Omega_2=\left(
\begin{array}{cc}
 -1.4969563\dots+ i* 1.2788564\dots & -1.8560003\dots - i*0.6569740\dots \\
-3.3529566\dots + i* 0.6218823\dots &  -1.1379124\dots +i*3.2146868\dots
\end{array}\right). 
\end{array}
$$
We apply the method described in
section 3, to obtain the monic polynomial 
$$
F_0(x)=x^6 - 54x^3 -27.  
$$
The coefficient $a_6$ is 1/12, so that $F_1(x)=1/12 F_0(x)$. 
The first prime for which the local factors of $C_{F_1}$ and $C_{-F_1}$ are
different is $p=67$.  Comparing with the polynomial 
$$
x^2(x+p/x-a_p)(x+p/x-^\sigma a_p),
$$
we see that the right sign is $-1$. We multiply $-F_1(x)$ by $6^2$ to obtain
an integral equation. We can finally assert that $A_f$ is 
the jacobian of the curve
$$
y^2=-3x^6+162x^3+81.
$$
The Igusa invariants of this curve are
$$
i_1=\frac{2^3\cdot 37^5}{3\cdot  7^3},\qquad i_2=-\frac{3\cdot  37^3\cdot  103}{2\cdot  7^3}\qquad
i_3=-\frac{5\cdot  37^2\cdot  881}{2^3\cdot  7^3}.
$$
We have also computed these Igusa invariants from the even Thetanullwerte associated to the period matrix $Z$, obtaining, of course, the same result.

Using Q. Liu's program, we find a minimal equation for the curve $C$:

$$
Y^2=X^6+54X^3-27,
$$
which is obtained from our equation through the change
$x=3/X$, $y=9 Y/X^3$, which corresponds essentially to a different ordering of
the modular forms $f_1,f_2$ as basis of $\langle f, {}^\sigma f\rangle$.

\section{Tables}

We present the equations  that we have obtained in the following table. We have labelled
the irreducible principally polarized two-dimensional factors $A_f$ of
$J_0(N)^{\mbox{new}}$ as $S_{NX}$. We have ordered the  two-dimensional
factors of $J_0(N)^{\mbox{new}}$ following the output of the {\tt Magma}
function {\tt SortDecomposition}. The letter   $X$  denotes the position of
$A_f$ with respect to this ordering. The third column indicates when we know that the given equation is minimal over $\Z[1/2]$.

\begin{center}
\setlongtables
\begin{longtable}{l@{\,\,}lc}
\hline\noalign{\smallskip}
 $A_f$  & $ \,\,\,C_F : y^2 \,=\, F(x),\quad  \Jac C_F\simeq A_f$ & minimal? \\[.1cm]
\hline\noalign{\smallskip}
\endfirsthead

\hline\noalign{\smallskip}
 $A_f$  & $ \,\,\,C_F : y^2 \,=\, F(x),\quad  \Jac C_F\simeq A_f$ & minimal? \\[.1cm]
\hline\noalign{\smallskip}
\endhead

\hline\noalign{\smallskip}
\endfoot
$S_{23A}$  &  $\,\,\,y^2 \,=x^6 - 8x^5 + 2x^4 + 2x^3 - 11x^2 + 10x - 7$  & yes \\[.1cm]
$S_{29A}$  &  $\,\,\,y^2 \,=x^6 - 4x^5 - 12x^4 + 2x^3 + 8x^2 + 8x - 7$& yes \\[.1cm]
$S_{31A}$  &  $\,\,\,y^2 \,= x^6 - 8x^5 + 6x^4 + 18x^3 - 11x^2 - 14x -3$& yes \\[.1cm]
$S_{63B}$  &  $\,\,\,y^2 \,=-3x^6 + 162x^3 + 81$&  \\[.1cm]
$S_{65B}$  &  $\,\,\,y^2 \,=-x^6 - 4x^5 + 3x^4 + 28x^3 - 7x^2 - 62x + 42$& yes \\[.1cm]
$S_{65C}$  &  $\,\,\,y^2 \,=-15x^6 + 36x^4 - 30x^3 + 72x^2-39$& yes \\[.1cm]
$S_{67B}$  &  $\,\,\,y^2 \,=x^6 + 2x^5 + x^4 - 2x^3 + 2x^2 - 4x + 1$& yes \\[.1cm]
$S_{73B}$  &  $\,\,\,y^2 \,=x^6 - 4x^5 + 2x^4 + 6x^3 + x^2 + 2x + 1 $& yes \\[.1cm]
$S_{87A}$  &  $\,\,\,y^2 \,=x^6 - 2x^4 - 6x^3 - 11x^2 - 6x - 3$& yes \\[.1cm]
$S_{93A}$  &  $\,\,\,y^2 \,=x^6 + 2x^4 - 6x^3 + 5x^2 + 6x + 1$& yes \\[.1cm]
$S_{103A}$  &  $\,\,\,y^2 \,=x^6 + 2x^4 + 2x^3 + 5x^2 + 6x + 1$& yes \\[.1cm]
$S_{107A}$  &  $\,\,\,y^2 \,=x^6 + 2x^5 + 5x^4 + 2x^3 - 2x^2 - 4x - 3$& yes \\[.1cm]
$S_{115B}$  &  $\,\,\,y^2 \,=x^6 + 2x^4 + 10x^3 + 5x^2 + 6x + 1$& yes \\[.1cm]
$S_{117B}$  &  $\,\,\,y^2 \,=x^6 - 10x^3 - 27$& yes \\[.1cm]
$S_{117C}$  &  $\,\,\,y^2 \,=-3x^6 - 12x^4 - 18x^3 - 48x^2-36x - 27$& yes \\[.1cm]
$S_{125A}$  &  $\,\,\,y^2 \,=x^6 + 2x^5 + 5x^4 + 10x^3 + 10x^2 + 8x + 1$& yes \\[.1cm]
$S_{125B}$  &  $\,\,\,y^2 \,=5x^6 - 10x^5 + 25x^4 - 50x^3 + 50x^2 - 40x + 5$& yes \\[.1cm]
$S_{133A}$  &  $\,\,\,y^2 \,=x^6 - 2x^5 + 5x^4 - 6x^3 + 10x^2 - 8x + 1$& yes \\[.1cm]
$S_{133B}$  &  $\,\,\,y^2 \,=-3x^6 - 22x^5 - 35x^4 + 50x^3 + 74x^2 - 100x + 29$& yes \\[.1cm]
$S_{135D}$  &  $\,\,\,y^2 \,=x^6 + 6x^4 - 10x^3 + 9x^2-30x - 11$& yes \\[.1cm]
$S_{147D}$  &  $\,\,\,y^2 \,=x^6 - 4x^4 + 2x^3 + 8x^2 - 12x + 9  $& yes \\[.1cm]
$S_{161B}$  &  $\,\,\,y^2 \,=x^6 + 6x^5 + 17x^4 + 22x^3 + 26x^2 + 12x + 1$& yes \\[.1cm]
$S_{167A}$  &  $\,\,\,y^2 \,=x^6 - 4x^5 + 2x^4 - 2x^3 - 3x^2 + 2x - 3$& yes \\[.1cm]
$S_{175E}$  &  $\,\,\,y^2 \,=x^6 + 2x^5 - 3x^4 + 6x^3 - 14x^2 + 8x - 3$& yes \\[.1cm]
$S_{177A}$  &  $\,\,\,y^2 \,=x^6 + 2x^4 - 6x^3 + 5x^2 - 6x + 1$& yes \\[.1cm]
$S_{177B}$  &  $\,\,\,y^2 \,=-15x^6 - 120x^5 - 530x^4 - 710x^3 - 515x^2-30x + 45$&  \\[.1cm]
$S_{188B}$  &  $\,\,\,y^2 \,=x^5 - x^4 + x^3 + x^2 - 2x + 1$& yes \\[.1cm]
$S_{189E}$  &  $\,\,\,y^2 \,=x^6 - 2x^3 - 27 $& yes \\[.1cm]
$S_{191A}$  &  $\,\,\,y^2 \,=x^6 + 2x^4 + 2x^3 + 5x^2 - 6x + 1 $& yes \\[.1cm]
$S_{205D}$  &  $\,\,\,y^2 \,=x^6 + 2x^4 + 10x^3 + 5x^2 - 6x + 1 $& yes \\[.1cm]
$S_{209B}$  &  $\,\,\,y^2 \,=x^6 - 4x^5 + 8x^4 - 8x^3 + 8x^2 + 4x + 4$& yes \\[.1cm]
$S_{213B}$  &  $\,\,\,y^2 \,=x^6 + 2x^4 + 2x^3 - 7x^2 + 6x - 3$& yes \\[.1cm]
$S_{221C}$  &  $\,\,\,y^2 \,=x^6 - 2x^5 + x^4 + 6x^3 + 2x^2 + 4x + 1$& yes \\[.1cm]
$S_{224C}$  &  $\,\,\,y^2 \,=-2x^6 - 8x^5 - 34x^4 - 48x^3 - 118x^2 + 56x + 154  $& yes \\[.1cm]
$S_{224D}$  &  $\,\,\,y^2 \,=2x^6 - 8x^5 + 34x^4 - 48x^3 + 118x^2 + 56x - 154  $& yes \\[.1cm]
$S_{243C}$  &  $\,\,\,y^2 \,=x^6 + 6x^3 - 27$& yes \\[.1cm]
$S_{250D}$  &  $\,\,\,y^2 \,=\,20\,x^6 - 140\,x^5 + 325\,x^4 +1050\,x^3 + 425\,x^2 + 160\,x + 80$ &   \\[.1cm]
$S_{256E}$  &  $\,\,\,y^2 \,=\,2\,x^5 - 128\,x $& yes \\[.1cm]
$S_{261A}$  &  $\,\,\,y^2 \,=x^6 - 6x^4 + 10x^3 + 21x^2-30x + 9$& yes \\[.1cm]
$S_{261B}$  &  $\,\,\,y^2 \,=-3x^6 + 18x^4 + 30x^3 - 63x^2 - 90x - 27$& yes \\[.1cm]
$S_{261D}$  &  $\,\,\,y^2 \,=-3x^6 + 6x^4 - 18x^3 + 33x^2 - 18x + 9$& yes \\[.1cm]
$S_{262C}$  &  $\,\,\,y^2 \,=-8x^5 + 56x^4 - 82x^3 - 312x^2 - 264x - 64$& yes \\[.1cm]
$S_{266B}$  &  $\,\,\,y^2 \,=\,8\,x^6 + 16\,x^5 + 13\,x^4 + 6\,x^3 - 19\,x^2 - 8\,x - 16 $& yes \\[.1cm]
$S_{268C}$  &  $\,\,\,y^2 \,=x^6 - 2x^5 + x^4 - 4x^3 + 2x^2 + 4x + 1$& yes \\[.1cm]
$S_{275G}$  &  $\,\,\,y^2 \,=-3x^6 - 2x^5 + x^4 - 14x^3 + 2x^2 - 8x + 1$& yes \\[.1cm]
$S_{279A}$  &  $\,\,\,y^2 \,=\,-3\,x^6 - 6\,x^4 - 18\,x^3 - 15\,x^2 + 18\,x - 3 $& yes \\[.1cm]
$S_{279B}$  &  $\,\,\,y^2 \,=\,-3\,x^6 + 6\,x^5 - 3\,x^4 - 6\,x^3 + 18\,x^2 - 12\,x + 9$& yes \\[.1cm]
$S_{287A}$  &  $\,\,\,y^2 \,=x^6 + 2x^5 - 3x^4 - 6x^3 - 10x^2 - 4x - 3$& yes \\[.1cm]
$S_{292A}$  &  $\,\,\,y^2 \,=-x^6 - 2x^5 - 4x^4 - 4x^3 - 3x^2 - 2x + 1$& yes \\[.1cm]
$S_{297E}$  &  $\,\,\,y^2 \,=\,x^6 - 12\,x^4 - 8\,x^3 + 12\,x^2 - 12\,x + 4$& yes \\[.1cm]
$S_{297F}$  &  $\,\,\,y^2 \,=\,-3\,x^6 + 36\,x^4 - 24\,x^3 - 36\,x^2 - 36\,x - 12 $& yes \\[.1cm]
$S_{299A}$  &  $\,\,\,y^2 \,=-3x^6 - 10x^5 - 7x^4 + 6x^3 + 6x^2 - 4x + 1$& yes \\[.1cm]
$S_{325H}$  &  $\,\,\,y^2 \,= \,-75\,x^6 + 180\,x^4 + 150\,x^3 + 360\,x^2 - 195 $& yes \\[.1cm]
$S_{335B}$  &  $\,\,\,y^2 \,=x^6 - 4x^5 - 48x^2 - 20x - 4$& yes \\[.1cm]
$S_{345G}$  &  $\,\,\,y^2 \,=\,x^6 - 12\,x^5 + 32\,x^4 + 24\,x^3 + 8\,x^2 - 12\,x + 4$& yes \\[.1cm]
$S_{351A}$  &  $\,\,\,y^2 \,=x^6 - 6x^4 + 18x^3 + 9x^2 - 18x + 5$ & yes \\[.1cm]
$S_{351C}$  &  $\,\,\,y^2 \,=\,-3\,x^6 + 18\,x^4 + 54\,x^3 - 27\,x^2 - 54\, x- 15$& yes \\[.1cm]
$S_{351D}$  &  $\,\,\,y^2 \,=\,21\,x^6 - 210\,x^5 + 525\,x^4 - 602\,x^3 + 714\,x^2 + 336\,x + 665  $&   \\[.1cm]
$S_{357E}$  &  $\,\,\,y^2 \,=x^6 + 8x^4 - 8x^3 + 20x^2 - 12x + 12$& yes \\[.1cm]
$S_{375C}$  &  $\,\,\,y^2 \,=\,105\,x^6 + 240\,x^5 + 550\,x^4 + 450\,x^3 + 325\,x^2 + 90\,x - 155  $& yes \\[.1cm]
$S_{376A}$  &  $\,\,\,y^2 \,=\,-x^5 - x^4 + 3\,x^3 + 3\,x^2 - 4\,x + 1$ & yes \\[.1cm]
$S_{376B}$  &  $\,\,\,y^2 \,=x^5 - x^3 + 2x^2 - 2x + 1$& yes \\[.1cm]
$S_{380D}$  &  $\,\,\,y^2 \,=x^5 - 7x^3 - 4x^2 + 5x + 5$& yes \\[.1cm]
$S_{387F}$  &  $\,\,\,y^2 \,=-12x^6 + 162x^3 + 324$&  \\[.1cm]
$S_{389B}$  &  $\,\,\,y^2 \,=x^6 + 10x^5 + 23x^4 - 20x^3 - 45x^2 + 46x - 11$& yes \\[.1cm]
$S_{391A}$  &  $\,\,\,y^2 \,=x^6 + 10x^4 - 6x^3 - 11x^2 + 18x - 7$& yes \\[.1cm]
$S_{424A}$  &  $\,\,\,y^2 \,=x^6 - 2x^5 + 6x^4 - 8x^3 + 10x^2 - 8x + 5  $& yes \\[.1cm]
$S_{440E}$  &  $\,\,\,y^2 \,=x^5 + 2x^3 - 11x^2 - 8x - 24$& yes \\[.1cm]
$S_{440G}$  &  $\,\,\,y^2 \,=x^5 - 2x^3 - 7x^2 - 8x + 8$& yes \\[.1cm]
$S_{441I}$  &  $\,\,\,y^2 \,=\,-3\,x^6 + 12\,x^4 + 6\,x^3 - 24\,x^2 - 36\,x - 27$& yes \\[.1cm]
$S_{464I}$  &  $\,\,\,y^2 \,=-x^6 - 2x^5 - 7x^4 - 6x^3 - 13x^2 - 4x - 8$& yes \\[.1cm]
$S_{476B}$  &  $\,\,\,y^2 \,=x^5 + 2x^4 + 3x^3 + 6x^2 + 4x + 1 $& yes \\[.1cm]
$S_{476D}$  &  $\,\,\,y^2 \,=x^5 - 2x^4 + 3x^3 - 6x^2 - 7$& yes \\[.1cm]
$S_{483C}$  &  $\,\,\,y^2 \,=\,x^6 + 12\,x^5 + 26\,x^4 - 34\,x^3 - 67\,x^2 + 90\,x - 27$ & yes \\[.1cm]
$S_{488,A}$  &  $\,\,\,y^2 \,=-3x^6 + 18x^5 - 27x^4 - 12x^3 - 27x^2-36x - 24$& yes \\[.1cm]
\end{longtable}
\end{center}

\providecommand{\bysame}{\leavevmode\hbox to3em{\hrulefill}\thinspace}

\end{document}